\documentclass[11pt,reqno]{amsart}

\usepackage{xspace}
\usepackage{amsmath,amsthm,amsfonts,amssymb,latexsym,mathrsfs,color}
\usepackage{hyperref}

\newtheorem{theorem}{Theorem}
\newtheorem{cor}[theorem]{Corollary}
\newtheorem{prop}[theorem]{Proposition}

\newtheorem{lemma}[theorem]{Lemma}

\newtheorem{example}[theorem]{Example}

\newcommand{\lpk}{{\rm lpk\,}}

\newcommand{\des}{{\rm des\,}}

\newcommand{\rbb}{\mathcal{RB}}

\newcommand{\msn}{\mathfrak{S}_n}

\newcommand{\rs}{\mathcal{RS}}
\newcommand{\rd}{\mathcal{RD}}

\newcommand{\rt}{\mathcal{RT}}

\newcommand{\lrf}[1]{\lfloor #1\rfloor}
\newcommand{\lrc}[1]{\lceil #1\rceil}

\title{The descent statistic on signed simsun permutations}
\author[S.-M.~Ma]{Shi-Mei Ma}
\address{School of Mathematics and Statistics,
        Northeastern University at Qinhuangdao,
         Hebei 066004, P.R. China}
\email{shimeimapapers@163.com (S.-M. Ma)}
\author[T. Mansour]{Toufik Mansour}
\address{D Department of Mathematics, University of Haifa, 3498838 Haifa, Israel}
\email{tmansour@univ.haifa.ac.il}
\author[H.-N.~Wang]{Hai-Na Wang}
\address{Department of Mathematics, Northeastern University, Shenyang, 110004, China}
\email{hainawangpapers@163.com}

\begin{document}

\maketitle
\begin{abstract}
In this paper we study the generating polynomials obtained by enumerating signed simsun permutations by number of
the descents. Properties of the polynomials, including the recurrence relations and generating functions are studied.
\end{abstract}
\date{\today}
\section{Introduction}
Let $\msn$ denote the symmetric group of all permutations of $[n]$, where $[n]=\{1,2,\ldots,n\}$.
Let $\pi=\pi(1)\pi(2)\cdots \pi(n)\in\msn$.
A {\it descent} in $\pi$ is an element $\pi(i)$ such that $\pi(i)>\pi(i+1)$, where $i\in [n-1]$.
%
We say that $\pi\in\msn$ has no {\it double descents} if there is no
index $i\in [n-2]$ such that $\pi(i)>\pi(i+1)>\pi(i+2)$.
A permutation $\pi\in\msn$ is called {\it simsun} if for all $k$, the
subword of $\pi$ restricted to $[k]$ (in the order
they appear in $\pi$) contains no double descents. For example,
$35142$ is simsun, but $35241$ is not. Let $\rs_n$ be the set of simsun permutations of length $n$.
Let $|C|$ denote the cardinality of a set
$C$. Simion and Sundaram~\cite[p.~267]{Sundaram1994} discovered that
$|\rs_n|=E_{n+1}$, where $E_n$ is the $n$th Euler number, which also is the number
alternating permutations in $\msn$ (see~\cite{Stanley} for instance).
Simsun permutations are a variant of the Andr\'e permutations of Foata and Sch\"utzenberger~\cite{Foata73}.
We refer the reader to~\cite{Branden11,Chow11,Deutsch12,Hetyei98,Ma16} for some recent results related to simsun permutations.

There have been extensive studies of the descent polynomials for simsun permutations (see~\cite{Chow11,Ma16} for instance).
Let
$S_n(x)=\sum_{\pi\in\rs_n}x^{\des_A(\pi)}$,
where $$\des_A(\pi)=|\{i\in [n-1]: \pi(i)>\pi(i+1)\}|.$$
It follows from~\cite[Theorem 3.2]{Sundaram1994} that the polynomials $S_n(x)$ satisfy the recurrence relation
\begin{equation*}\label{Snx-recu}
S_{n+1}(x)=(1+nx)S_n(x)+x(1-2x)S_n'(x),
\end{equation*}
with $S_0(x)=1$. Let $RS(x,z)=\sum_{n\geq 0}S_n(x)\frac{z^n}{n!}$.
Chow and Shiu~\cite[Theorem~2.1]{Chow11} obtained that
\begin{equation*}\label{Sxz-expon}
RS(x,z)=\left(\frac{\sqrt{2x-1}\sec\left(\frac{z}{2}\sqrt{2x-1}\right)}
{\sqrt{2x-1}-\tan\left(\frac{z}{2}\sqrt{2x-1}\right)}\right)^2.
\end{equation*}
Recall that the classical {\it Eulerian polynomials} of type $A$ are defined by
$A_n(x)=\sum_{\pi\in\msn}x^{\des_A(\pi)}$.
From~\cite[Proposition~2.7]{Foata73}, we have
\begin{equation*}\label{AnxSnk}
S_n\left(\frac{2x}{(1+x)^2}\right)=\frac{A_{n+1}(x)}{(1+x)^n}.
\end{equation*}
A {\it left peak} in $\pi$ is an index $i\in[n-1]$ such that $\pi(i-1)<\pi(i)>\pi(i+1)$, where we take $\pi(0)=0$.
Let $\lpk(\pi)$ denote the number of
left peaks in $\pi$. For example, $\lpk(21435)=2$.
It is clear that every descent of a simsun permutation is a left peak. Hence $S_n(x)=\sum_{\pi\in\rs_n}x^{\lpk(\pi)}$.
Let
$\widehat{W}_n(x)=\sum_{\pi\in\msn}x^{\lpk(\pi)}$.
There is a close connection between $S_n(x)$ and $\widehat{W}_n(x)$ (see~\cite[Eq.~(6)]{Ma16}):
\begin{equation*}\label{SnxWnx}
S_n(x)=\frac{1}{2^n}\sum_{k=0}^n\binom{n}{k}\widehat{W}_k(2x)\widehat{W}_{n-k}(2x).
\end{equation*}

It is now well known that simsun permutations and
signed simsun permutations are useful in computing the $cd$-index of the Boolean algebra and the cubical lattice, respectively  (see~\cite{Billera97,Ehrenborg98}).
Let $B_n$ be the hyperoctahedral group of rank $n$. Elements $\pi$ of $B_n$ are signed permutations of the set $\pm[n]$ such that $\pi(-i)=-\pi(i)$ for all $i$, where $\pm[n]=\{\pm1,\pm2,\ldots,\pm n\}$.
Throughout this paper, we always identify a signed permutation $\pi=\pi(1)\cdots\pi(n)$ with the word $\pi(0)\pi(1)\cdots\pi(n)$, where $\pi(0)=0$.
A double descent of $\pi\in B_n$ is an index $i\in [n-1]$ such that
$\pi(i-1)>\pi(i)>\pi(i+1)$.
Let $R(\pm[k])$ be the set of signed permutations of $B_k$ with no double descents.
Following~\cite[Section 7]{Ehrenborg98},
a {\it signed simsun permutation} $\pi$ of length $n$ is a permutation of $B_n$
such that for all $0\leq k\leq n$,
if we remove the $k$ entries $\pm n, \pm (n-1),\ldots, \pm (n-k+1)$ from $\pi$,
the resulting permutation belongs to $R(\pm [n-k])$.
Let $\rbb_n$ denote the set of signed simsun permutations of $B_n$.
For example, $01(-3)2(-5)4$ is a signed simsun permutations, while $\pi'=01(-3)2(-6)(-4)(-5)$ is not, since when we remove $\pm6$ from $\pi'$, the resulting permutation $01(-3)2(-4)(-5)$ contains a double descent. In particular,
$$\rbb_2=\{012,01(-2),021,0(-2)1,0(-1)2,02(-1),0(-2)(-1)\}.$$

Denote by $D_n$ the set of even-signed permutations of $B_n$, i.e., for any $\pi\in D_n$, the set
$\{\pi(1),\pi(2),\ldots,\pi(n)\}$ contains an even number of negative terms.
For $n\geq 2$, $D_n$ forms a normal subgroup of $B_n$ of index 2.
Set $T_n=B_n\setminus D_n$.
Following~\cite{Bre94}, the descent statistic of type $B$ is defined by
$$\des_B(\pi)=|\{i\in\{0,1,2,\ldots,n-1\}:\pi(i)>\pi({i+1})\}|.$$
Let $$E(n,k)=|\{\pi\in D_n: \des_B(\pi)=k\}|,~\widetilde{E}(n,k)=|\{\pi\in T_n: \des_B(\pi)=k\}|.$$
Very recently, Borowiec and Mlotkowski~\cite{Borowiec16} studied the type $D$ Eulerian numbers $E(n,k)$ and $\widetilde{E}(n,k)$. In particular, they discovered a remarkable formula (see~\cite[Proposition 4.3]{Borowiec16}):
\begin{equation}\label{Enk}
E(n,k)-\widetilde{E}(n,k)=(-1)^k\binom{n}{k}.
\end{equation}
Let $\rd_n$ and $\rt_n$ denote the sets of simsun permutations of $D_n$ and $T_n$, respectively.
Let $$D(n,k)=|\{\pi\in \rd_n: \des_B(\pi)=k\}|,~T(n,k)=|\{\pi\in \rt_n: \des_B(\pi)=k\}|.$$
As a correspondence of~\ref{Enk}, we recently observed the following formula:
\begin{equation}\label{Dnk}
D(n,k)-T(n,k)=(-1)^k\binom{n-k+1}{k},
\end{equation}
which is implied by Theorem~\ref{thm02} of Section~\ref{Section-3}. Motivated by~\eqref{Dnk}, we shall study
the polynomials obtained by enumerating permutations of $\rbb_n,\rd_n$ and $\rt_n$ by number of the descents.
\section{On signed simsun permutations of $B_n$ }\label{Section-2}
\hspace*{\parindent}
Let $\rbb_n^{+}=\{\pi\in\rbb_n: \pi(1)>0\}$ and $\rbb_n^{-}=\{\pi\in\rbb_n: \pi(1)<0\}$.
We define
\begin{align*}
R^{+}_n(x)&=\sum_{\pi\in\rbb_n^{+}}x^{\des_B(\pi)}=\sum_{k\geq 0} R^{+}(n,k)x^k,\\
R^{-}_n(x)&=\sum_{\pi\in\rbb_n^{-}}x^{\des_B(\pi)}=\sum_{k\geq 0} R^{-}(n,k)x^k,\\
R_n(x)&=\sum_{\pi\in\rbb_n}x^{\des_B(\pi)}=\sum_{k\geq 0} R(n,k)x^k.
\end{align*}
It is clear that $R(n,k)=R^{+}(n,k)+R^{-}(n,k)$.
The following lemma is a fundamental result.
\begin{lemma}\label{lemma1}
For $n\geq 2$, we have
\begin{equation}\label{Rnk-recu01}
R^+(n,k)=2kR^+(n-1,k)+(2n-4k+2)R^+(n-1,k-1)+R(n-1,k),
\end{equation}
\begin{equation}\label{Rnk-recu02}
R^-(n,k)=2kR^-(n-1,k)+(2n-4k+3)R^-(n-1,k-1)+R(n-1,k-1),
\end{equation}
\begin{equation}\label{Rnk-recu03}
R(n,k)=(2k+1)R(n-1,k)+(2n-4k+3)R(n-1,k-1)+R^-(n-1,k-1).
\end{equation}
\end{lemma}
\begin{proof}
Define
$$\rbb^+_{n,k}=\{\pi\in\rbb^+_n: \des(\pi)=k\},$$
$$\rbb^-_{n,k}=\{\pi\in\rbb^-_n: \des(\pi)=k\}.$$
Firstly, we prove~\eqref{Rnk-recu01}.
In order to get a permutation $\pi'\in\rbb^{+}_{n,k}$ from a permutation $\pi\in\rbb_{n-1}$,
we distinguish among the following cases:
\begin{enumerate}
\item [(i)] If $\pi\in \rbb_{n-1,k}^+$, then we can insert the entry $n$ right after a descent or insert the entry $-n$ right after a descent of $\pi$. Moreover, we can put the entry $n$ at the end of $\pi$.
This accounts for $(2k+1)R^+(n-1,k)$ possibilities.
\item [(ii)] If $\pi\in \rbb_{n-1,k-1}^+$, then we cannot insert the entry $n$ immediately before or right after each descent of $\pi$ and we cannot insert the entry $n$ at the end of $\pi$. Moreover, we cannot insert the entry $-n$ immediately before or right after each descent bottom and we also cannot insert $-n$ right after $\pi(0)$, where a descent bottom is an entry $\pi(i)$ such that $\pi(i-1)>\pi(i)$ and $1\leq i\leq n-1$. Hence there are $n-2k+1$ positions could be inserted the entry $n$ or $-n$.
    This accounts for $(2n-4k+2)R^+(n-1,k-1)$ possibilities.
\item [(iii)] If $\pi\in \rbb_{n-1,k}^-$, then we have to put $n$ right after $\pi(0)$. This accounts for $R^-(n-1,k)$ possibilities.
 \end{enumerate}
Therefore,
\begin{align*}
R^+(n,k)&=(2k+1)R^+(n-1,k)+(2n-4k+2)R^+(n-1,k-1)+R^-(n-1,k)\\
&=2kR^+(n-1,k)+(2n-4k+2)R^+(n-1,k-1)+R(n-1,k).
\end{align*}

Secondly, we prove~\eqref{Rnk-recu02}. In order to get a permutation $\pi'\in\rbb_{n,k}^{-}$ from a permutation $\pi\in\rbb_{n-1}$,
we also distinguish among the following cases:
\begin{enumerate}
\item [(i)] If $\pi\in \rbb_{n-1,k}^-$, then we can insert the entry $n$ right after any descent except the first descent or insert the entry $-n$ right after any descent. Moreover, we can insert $n$ at the end of $\pi$. This accounts for $2kR^-(n-1,k)$ possibilities.
\item [(ii)] If $\pi\in \rbb_{n-1,k-1}^-$, we cannot insert the entry $n$ immediately before or right after any descent and we cannot insert the entry $-n$ immediately before or right after each descent bottom. Moreover, we cannot put $n$ at the end of $\pi$. Note that $\pi(0)$ is a descent.
    Hence there are $n-1-1-2(k-2)=n-2k+2$ positions could be inserted the entry $n$ and there are $n-2(k-1)=n-2k+2$ positions could be inserted the entry $-n$. This accounts for $(2n-4k+4)R^-(n-1,k-1)$ possibilities.
\item [(iii)]If $\pi\in \rbb_{n-1,k-1}^+$, then we have to put the entry $-n$ right after $\pi(0)$. This accounts for $R^+(n-1,k-1)$ possibilities.
 \end{enumerate}
Therefore,
\begin{align*}
R^-(n,k)&=2kR^-(n-1,k)+(2n-4k+4)R^-(n-1,k-1)+R^+(n-1,k-1)\\
&=2kR^-(n-1,k)+(2n-4k+3)R^-(n-1,k-1)+R(n-1,k-1).
\end{align*}
Finally,
combining~\eqref{Rnk-recu01} and~\eqref{Rnk-recu02}, we immediately get~\eqref{Rnk-recu03}.
\end{proof}

So the following proposition is immediate.
\begin{prop}\label{prop1}
The polynomials $R_{n}^+(x),R_{n}^-(x)$ and $R_{n}(x)$ satisfy the following recurrence relations
\begin{align*}
R_{n+1}^+(x)&=2nxR_{n}^+(x)+2x(1-2x)\frac{d}{dx}R_{n}^+(x)+R_{n}(x),\\
R_{n+1}^-(x)&=(2n+1)xR_{n}^-(x)+2x(1-2x)\frac{d}{dx}R_{n}^-(x)+xR_{n}(x),\\
R_{n+1}(x)&=(1+(2n+1)x)R_{n}(x)+2x(1-2x)\frac{d}{dx}R_{n}(x)+xR_{n}^-(x),
\end{align*}
with the initial conditions $R_{0}^+(x)=R_0(x)=1$ and $R_{0}^-(x)=0$.
\end{prop}

By Proposition~\ref{prop1}, it is easy to verify that
$\deg R^{+}_n(x)=\lrf{n/2}$ and $\deg R^{-}_n(x)=\deg R_n(x)=\lrc{n/2}$.
For convenience, we list the first few terms of $R^{+}_n(x),R^{-}_n(x)$ and $R_n(x)$:
\begin{align*}
R^{+}_1(x)& =1,R^{+}_2(x)=1+3x,R^{+}_3(x)=1+16x,R^{+}_4(x)=1+61x+41x^2;\\
R^{-}_1(x)& =x,R^{-}_2(x) =3x,R^{-}_3(x)=7x+9x^2,R^{-}_4(x)=15x+80x^2;\\
R_1(x)& =1+x,R_2(x) =1+6x,R_3(x)=1+23x+9x^2,R_4(x)=1+76x+121x^2.
\end{align*}

Define
\begin{align*}
R^+(x;t)&=\sum_{n\geq0}R^+_n(x)\frac{t^n}{n!};\\
R^-(x;t)&=\sum_{n\geq0}R^-_n(x)\frac{t^n}{n!};\\
R(x;t)&=\sum_{n\geq0}R_n(x)\frac{t^n}{n!}.
\end{align*}
Clearly, $R(x;t)=R^+(x;t)+R^-(x;t)$.
We can now conclude the first main result of this paper.
\begin{theorem}\label{thRR1}
Let $u(x,t)=t\sqrt{2x-1}-\arctan(\sqrt{2x-1})$. The
generating functions $R^+(x;t)$ and $R^-(x;t)$ are respectively given by
\begin{align*}
&R^+(x;t)=\frac{\sqrt{2x-1}}{2\sqrt{x}p(x)}
\left(p^2(x)F_2(u(x,t))+F_1(u(x,t))\right),\\
&R^-(x;t)=\frac{-\sqrt{2x-1}}{2p(x)}
\left(p^2(x)F_2(u(x,t))-F_1(u(x,t))\right),
\end{align*}
where
$p(x)=\left(\frac{\sqrt{2x}+1}{\sqrt{2x}-1}\right)^{\frac{\sqrt{2}}{4}}$
and
\begin{align*}
F_1(x)=\frac{-1}{\sqrt{2}\sin(x)}\left(\frac{-\sin(x)}{1+\cos(x)}\right)^{-\frac{1}{\sqrt{2}}},\quad
F_2(x)=\frac{-1}{\sqrt{2}\sin(x)}\left(\frac{-\sin(x)}{1+\cos(x)}\right)^{\frac{1}{\sqrt{2}}}.
\end{align*}
\end{theorem}
\begin{proof}
Note that $R(x;t)=R^+(x;t)+R^-(x;t)$.
By writing the statement of
Proposition~\ref{prop1} in terms of generating functions
$R^+(x;t)$ and $R^-(x;t)$, we obtain
\begin{align*}
(1-2xt)\frac{d}{dt}R^+(x;t)&=2x(1-2x)\frac{d}{dx}R^+(x;t)+R^+(x;t)+R^-(x;t),\\
(1-2xt)\frac{d}{dt}R^-(x;t)&=2x(1-2x)\frac{d}{dx}R^-(x;t)+xR^+(x;t)+2xR^-(x;t).
\end{align*}
In order to solve this system of partial differential equations,
let us define $$M(x;t)=a(x)R^+(x;t)+b(x)R^-(x;t),$$ where $a(x)$ and
$b(x)$ are functions on $x$. It follows that $M(x;t)$ satisfies
\begin{align*}
(1-2xt)\frac{d}{dt}M(x;t)&=2x(1-2x)\frac{d}{dx}M(x;t)\\
&+\left(a(x)+xb(x)-2x(1-2x)\frac{d}{dx}a(x)\right)R^+(x;t)\\
&+\left(2xb(x)+a(x)-2x(1-2x)\frac{d}{dx}b(x)\right)R^-(x;t).
\end{align*}
Now, let us assume that the functions $a(x)$ and $b(x)$ satisfies
the following system of differential equations:
\begin{align*}
&a(x)+xb(x)-2x(1-2x)\frac{d}{dx}a(x)=0,\\
&2xb(x)+a(x)-2x(1-2x)\frac{d}{dx}b(x)=0.
\end{align*}
The system has the following general solution
$$a(x)=\frac{\sqrt{x}}{\sqrt{2x-1}}(c/p(x)+dp(x)),\quad
b(x)=\frac{1}{\sqrt{2x-1}}(-c/p(x)+dp(x)),$$
where
$$p(x)=\left(\frac{\sqrt{2x}+1}{\sqrt{2x}-1}\right)^{\frac{\sqrt{2}}{4}}.$$
From now, let us assume that the pair of the functions
$(a(x),b(x))$ are either
$(a_1(x),b_1(x))=\left(\frac{\sqrt{x}p(x)}{\sqrt{2x-1}},\frac{p(x)}{\sqrt{2x-1}}\right)$
or
$(a_2(x),b_2(x))=\left(\frac{\sqrt{x}}{\sqrt{2x-1}p(x)},\frac{-1}{\sqrt{2x-1}p(x)}\right)$.
Hence, the generating function $M(x;t)=M_{a_j,b_j}(x;t)$ for a pair
$(a_j(x),b_j(x))$, $j=1,2$, satisfies
\begin{align*}
(1-2xt)\frac{d}{dt}M_{a_j,b_j}(x;t)&=2x(1-2x)\frac{d}{dx}M_{a_j,b_j}(x;t)=0,
\end{align*}
which implies that there exist functions $F_1$ and $F_2$ such
that
$$M_{a_j,b_j}(x;t)=F_j(t\sqrt{2x-1}-\arctan(\sqrt{2x-1})).$$
Since $R^+(x;0)=R^+_0(x)=1$ and $R^-(x;0)=R^-_0(x)=0$, we see that
$M_{a_j,b_j}(x;0)=a_j(x)=F_j(-\arctan(\sqrt{2x-1}))$. Thus,
\begin{align*}
F_1(x)&=-\frac{p((\tan^2(x)+1)/2)}{\sqrt{2}\sin(x)}
=\frac{-1}{\sqrt{2}\sin(x)}\left(\frac{-\sin(x)}{1+\cos(x)}\right)^{\frac{-1}{\sqrt{2}}},\\
F_2(x)&=-\frac{p((\tan^2(x)+1)/2)}{\sqrt{2}\sin(x)}
=\frac{-1}{\sqrt{2}\sin(x)}\left(\frac{-\sin(x)}{1+\cos(x)}\right)^{\frac{1}{\sqrt{2}}}.
\end{align*}
Therefore,
\begin{align*}
a_1(x)R^+(x;t)+b_1(x)R^-(x;t)&=F_1(u(x,t)),\quad a_2(x)R^+(x;t)+b_2(x)R^-(x;t)&=F_2(u(x,t)),
\end{align*}
which implies
\begin{align*}
R^+(x;t)&=\frac{\sqrt{2x-1}}{2\sqrt{x}p(x)}\left(p^2(x)F_2(u(x,t))+F_1(u(x,t))\right),\\
R^-(x;t)&=\frac{-\sqrt{2x-1}}{2p(x)}(p^2(x)F_2(u(x,t))-F_1(u(x,t))).
\end{align*}
This completes the proof.
\end{proof}

Here we give two examples as applications of Theorem~\ref{thRR1}.
\begin{example}
Let $r_n=|\rbb_n|=\sum_{k\geq 0}R(n,k)$.
In order to get the first few terms of $r_n$,
we need to expand the
generating function $R(1;t)$. Let $\alpha=\sqrt{2}$. Then for $x=1$, we get
\begin{align*}
R(1;t)&=\left(\frac{\alpha+1}{\alpha-1}\right)^{-\alpha/4}\frac{1}{\cos t-\sin t}\left(\frac{\sin(t-\pi/4)}{\cos(t-\pi/4)-1}\right)^{\alpha/2}\\
&=\frac{(\cos t-\sin t)^{\alpha/2-1}}{(\alpha-\sin t-\cos t)^{\alpha/2}(1+\alpha)^{\alpha/2}}\\
&=\frac{1}{(\alpha^2-1)^{\alpha/2}}\left(1+\frac{\alpha^2-4\alpha+2}{2(1-\alpha)}t
+\frac{\alpha^4-12\alpha^3+34\alpha^2-32\alpha+12}{8(1-\alpha)^2}t^2+\cdots\right)\\
&=1+2t+7\frac{t^2}{2!}+33\frac{t^3}{3!}+198\frac{t^4}{4!}+1439\frac{t^5}{5!}+12291\frac{t^6}{6!}
+120622\frac{t^7}{7!}+\cdots.
\end{align*}
\end{example}

\begin{example}
Let $s_n=\sum_{k\geq0}kR(n,k)$.
In order to get the first few terms of $s_n$,
we need to expand the
generating function $R'(1;t)=\frac{d}{dx}R(x;t)\mid_{x=1}$. Let $\alpha=\sqrt{2}$. Then by Theorem~\ref{thRR1}, we obtain
\begin{align*}
R'(1;t)&=\left( {\frac{\alpha-1}{\alpha+1}} \right) ^{\alpha/4} \frac{ (3+4t)\cos t +(4t-7)\sin t +4t-2}{4(1-\sin(2t))} \left({\frac {\cos t-\sin t}{\alpha+\sin
 t +\cos t }} \right)^{-\alpha/2}\\
 &+\left( {\frac{\alpha+1}{\alpha-1}} \right) ^{\alpha/4} \frac{\sin t -\cos t}{4(1-\sin(2t))}
 \left( {\frac {\cos t -\sin t}{\alpha+\sin t +\cos t }} \right) ^{\alpha/2}\\
 &=t+6\frac{t^2}{2!}+41\frac{t^3}{3!}+318\frac{t^4}{4!}+2840\frac{t^5}{5!}+28736\frac{t^6}{6!}
 +325991\frac{t^7}{7!}+\cdots.
\end{align*}
\end{example}
\section{On even-signed simsun permutations}\label{Section-3}
\hspace*{\parindent}
Let
\begin{align*}
\rd_n^{+}&=\{\pi\in\rd_n: \pi(1)>0\}, ~\rd_n^{-}=\{\pi\in\rd_n: \pi(1)<0\},\\
\rt_n^{+}&=\{\pi\in\rt_n: \pi(1)>0\}, ~\rt_n^{-}=\{\pi\in\rt_n: \pi(1)<0\}.
\end{align*}

Define
\begin{align*}
D^{+}_n(x)&=\sum_{\pi\in\rd_n^{+}}x^{\des_B(\pi)}=\sum_{k\geq 0} D^{+}(n,k)x^k,~
D^{-}_n(x)=\sum_{\pi\in\rd_n^{-}}x^{\des_B(\pi)}=\sum_{k\geq 0} D^{-}(n,k)x^k,\\
T^{+}_n(x)&=\sum_{\pi\in\rt_n^{+}}x^{\des_B(\pi)}=\sum_{k\geq 0} T^{+}(n,k)x^k,~
T^{-}_n(x)=\sum_{\pi\in\rt_n^{-}}x^{\des_B(\pi)}=\sum_{k\geq 0} T^{-}(n,k)x^k,\\
D_n(x)&=\sum_{\pi\in\rd_n}x^{\des_B(\pi)}=\sum_{k\geq 0} D(n,k)x^k,~
T_n(x)=\sum_{\pi\in\rt_n}x^{\des_B(\pi)}=\sum_{k\geq 0} T(n,k)x^k.
\end{align*}
Clearly,
$R^+(n,k)=D^{+}(n,k)+T^{+}(n,k),
R^-(n,k)=D^{-}(n,k)+T^{-}(n,k)$ and
$R(n,k)=D(n,k)+T(n,k)$.

\begin{lemma}\label{lemma2}
We have
\begin{align*}
D^{+}(n,k)&=kR^+(n-1,k)+(n-2k+1)R^+(n-1,k-1)+D(n-1,k);\\
T^{+}(n,k)&=kR^+(n-1,k)+(n-2k+1)R^+(n-1,k-1)+T(n-1,k);\\
D^{-}(n,k)&=kR^-(n-1,k)+(n-2k+2)R^{-}(n-1,k-1)+T^{+}(n-1,k-1);\\
T^{-}(n,k)&=kR^-(n-1,k)+(n-2k+2)R^-(n-1,k-1)+{D}^{+}(n-1,k-1);\\
D(n,k)&=(1+k)D(n-1,k)+(n-2k+1)D(n-1,k-1)+kT(n-1,k)+\\
&(n-2k+2)T(n-1,k-1)+D^{-1}(n-1.k-1);\\
T(n,k)&=(1+k)T(n-1,k)+(n-2k+1)T(n-1,k-1)+k{D}(n-1,k)+\\
&(n-2k+2){D}(n-1,k-1)+T^{-1}(n-1,k-1).
\end{align*}
\end{lemma}
\begin{proof}
Define
$$\rd^+_{n,k}=\{\pi\in\rd^+_n: \des(\pi)=k\},~\rd^-_{n,k}=\{\pi\in\rd^-_n: \des(\pi)=k\};$$
$$\rt^+_{n,k}=\{\pi\in\rt^+_n: \des(\pi)=k\},\rt^-_{n,k}=\{\pi\in\rt^-_n: \des(\pi)=k\}.$$
We only prove the first recurrence relation and the others can be proved in the same way.

In order to get a permutation $\pi'\in\rd^+_{n,k}$ from a permutation $\pi\in\rbb_{n-1}$,,
we distinguish among the following cases:
\begin{enumerate}
\item [(i)] If $\pi\in \rd^+_{n-1,k}$, then we can insert the entry $n$ right after any descent. Moreover, we can also put $n$ at the end of $\pi$. This accounts for $(k+1)D^+(n-1,k)$ possibilities.
\item [(ii)] If $\pi\in \rd^+_{n-1,k-1}$, we cannot insert $n$ immediately before or right after any descent, and we cannot put $n$ at the end of $\pi$. Thus $n$ can be inserted into the remaining $n-1-2(k-1)=n-2k+1$ positions. This accounts for $(n-2k+1)D^+(n-1,k-1)$ possibilities.
\item [(iii)] If $\pi\in \rd^-_{n-1,k}$, then we have to insert $n$ right after $\pi(0)$. This accounts for $D^-(n-1,k)$ possibilities.
\item [(iv)] If $\pi\in \rt^+_{n-1,k}$, then we can insert $-n$ right after any descent. This accounts for $kT^+(n-1,k)$ possibilities.
\item [(v)] If $\pi\in \rt^+_{n-1,k-1}$, then we cannot insert $-n$ immediately before or right after any descent bottom. Moreover, we cannot put $-n$ right after $\pi(0)$. Thus the entry $-n$ can be inserted into the remaining $n-1-2(k-1)=n-2k+1$ positions. This accounts for $(n-2k+1)T^+(n-1,k-1)$ possibilities.
 \end{enumerate}

Therefore,
\begin{align*}
D^{+}(n,k)&=(k+1)D^+(n-1,k)+(n-2k+1)D^+(n-1,k-1)+D^-(n-1,k)+\\
&kT^+(n-1,k)+(n-2k+1)T^+(n-1,k-1)\\
&=k(D^+(n-1,k)+T^+(n-1,k))+(n-2k+1)(D^+(n-1,k-1)+T^+(n-1,k-1))\\
&+D^+(n-1,k)+D^-(n-1,k)\\
&=kR^+(n-1,k)+(n-2k+1)R^+(n-1,k-1)+D(n-1,k),
\end{align*}
and this completes the proof.
\end{proof}

By Lemma~\ref{lemma2}, we immediately get the following proposition.
\begin{prop}\label{prop2}
Set $D_{0}^+(x)=D_0(x)=1,D_{0}^-(x)=T_0(x)=T_{0}^-(x)=T_{0}^+(x)=0$.
Then for $n\geq 0$, we have the following recurrence relations:
\begin{align*}
D^{+}_{n+1}(x)&=nxR^{+}_n(x)+x(1-2x)\frac{d}{dx}R^{+}_n(x)+D_{n}(x);\\
T^{+}_{n+1}(x)&=nxR^{+}_n(x)+x(1-2x)\frac{d}{dx}R^{+}_n(x)+T_{n}(x);\\
D^{-}_{n+1}(x)&=(n+1)xR^{-}_n(x)+x(1-2x)\frac{d}{dx}R^{-}_n(x)+xT_{n}^+(x);\\
T^{-}_{n+1}(x)&=(n+1)xR^{-}_n(x)+x(1-2x)\frac{d}{dx}R^{-}_n(x)+x{D}_{n}^+(x);\\
D_{n+1}(x)&=(1+nx)D_n(x)+(n+1)xT_n(x)+x(1-2x)\frac{d}{dx}R_n(x)+xD^{-}_{n}(x);\\
T_{n+1}(x)&=(1+nx)T_n(x)+(n+1)x{D}_n(x)+x(1-2x)\frac{d}{dx}R_n(x)+xT^{-}_{n}(x).
\end{align*}
\end{prop}


The {\it Chebyshev polynomials of the second kind} may be recursively defined by
$U_0(x)=1,U_1(x)=2x$ and $U_{n+1}(x)=2xU_n(x)-U_{n-1}(x)$ for $n\geq 1$.
A well known explicit formula is the following:
$$U_n(x)=\sum_{k=0}^{\lrf{n/2}}(-1)^k\binom{n-k}{k}(2x)^{n-2k}.$$

Now we present the second main result of this paper.
\begin{theorem}\label{thm02}
For $n\geq 0$, we have
\begin{equation}\label{fibonacc}
D_n(x)-T_n(x)=D^+_{n+1}(x)-T^+_{n+1}(x)=\frac{1}{x}\left(T^-_{n+2}(x)-D^-_{n+2}(x)\right)=x^{\frac{n+1}{2}}U_{n+1}\left(\frac{1}{2\sqrt{x}}\right).
\end{equation}
\end{theorem}
\begin{proof}
Using Proposition~\ref{prop2}, we get
$$D_n(x)-T_n(x)=D^+_{n+1}(x)-T^+_{n+1}(x)=\frac{1}{x}\left(T^-_{n+2}(x)-D^-_{n+2}(x)\right),$$
and
$D_{n}(x)-T_{n}(x)=(1-x)(D_{n-1}(x)-T_{n-1}(x))-x(T^-_{n-1}(x)-D^-_{n-1}(x))$.
Note that
$$x^{\frac{n+1}{2}}U_{n+1}\left(\frac{1}{2\sqrt{x}}\right)=\sum_{k\geq 0}(-1)^k\binom{n-k+1}{k}x^k.$$
If $n\leq 3$, the equality~\eqref{fibonacc} is obvious, so we proceed to the inductive step. Assume that $n\geq 4$ and that the
equality holds for $m\leq n-1$.
Then
\begin{align*}
D_{m+1}(x)-T_{m+1}(x)&=(1-x)(D_{m}(x)-T_{m}(x))-x(T^-_{m}(x)-D^-_{m}(x))\\
&=(1-x)\sum_{k\geq 0}(-1)^k\binom{m-k+1}{k}x^k-x^2\sum_{k\geq 0}(-1)^k\binom{m-k-1}{k}x^k\\
&=\sum_{k\geq 0}(-1)^k\left(\binom{m-k+1}{k}+\binom{m-k+2}{k-1}-\binom{m-k+1}{k-2}\right)x^k\\
&=\sum_{k\geq 0}(-1)^k\left(\binom{m-k+1}{k}+\frac{m-2k+3}{k-1}\binom{m-k+1}{k-2}\right)x^k\\
&=\sum_{k\geq 0}(-1)^k\frac{(m-k+1)!}{k!(m-2k+2)!}\left(m-k+2\right)x^k\\
&=\sum_{k\geq 0}(-1)^k\binom{m-k+2}{k}x^k,
\end{align*}
as desired.
\end{proof}

Combining~\eqref{fibonacc} and the fact that $R_n(x)=D_n(x)+T_n(x)$, so the following is immediate.
\begin{cor}
We have
\begin{align*}
D(n,k)&=\frac{1}{2}R(n,k)+\frac{1}{2}\binom{n-k+1}{k}(-1)^k;\\
T(n,k)&=\frac{1}{2}R(n,k)-\frac{1}{2}\binom{n-k+1}{k}(-1)^k.
\end{align*}
\end{cor}

The {\it Fibonacci sequence} is defined recursively by $F_0=0, F_1=1$ and $F_{n}=F_{n-1}+F_{n-2}$ for $n\geq 2$ (see~\cite[A000045]{Sloane}).
A well known sum formula for $F_n$ is the following:
$$F_{n+1}=\sum_{k=0}^{\lrf{n/2}}\binom{n-k}{k}.$$
As a special case of~\eqref{fibonacc}, we get the following.
\begin{cor}
Let $d_n=D_n(-1)$ and $t_n=T_n(-1)$. We have
$d_n-t_n=F_{n+2}$.
\end{cor}

For convenience, we end this paper by providing the first few terms of the polynomials discussed in this section:
\begin{align*}
D^{+}_1(x)& =1,D^{+}_2(x)=1+x,D^{+}_3(x)=1+7x,D^{+}_4(x)=1+29x+21x^2;\\
T^{+}_1(x)& =0,T^{+}_2(x)=2x,T^{+}_3(x)=9x,T^{+}_4(x)=32x+20x^2;\\
D^{-}_1(x)& =0,D^{-}_2(x) =x,D^{-}_3(x)=3x+5x^2,D^{-}_4(x)=7x+41x^2;\\
T^{-}_1(x)& =x,T^{-}_2(x) =2x,T^{-}_3(x)=4x+4x^2,T^{-}_4(x)=8x+39x^2;\\
D_1(x)& =1,D_2(x)=1+2x,D_3(x)=1+10x+5x^2,D_4(x)=1+36x+62x^2;\\
T_1(x)& =x,T_2(x)=4x,T_3(x)=13x+4x^2,T_4(x)=40x+59x^2.
\end{align*}

\end{document}